\documentclass{elsarticle}

\usepackage{amsthm}
\usepackage{amsmath}
\usepackage{amssymb}
\usepackage{graphicx}
\usepackage{amscd}

\theoremstyle{plain}
\newtheorem{thm}{Theorem}[section]
\newtheorem{lem}[thm]{Lemma}
\newtheorem{prop}[thm]{Proposition}
\newtheorem{cor}[thm]{Corollary}

\theoremstyle{definition}

\theoremstyle{remark}
\newtheorem*{rem}{Remark}

\begin{document}

\title{A Converse to a Theorem on Normal Forms of Volume Forms with Respect to a Hypersurface}
\author{Konstantinos Kourliouros}
\address{Imperial College London, Department of Mathematics,  Huxley Building 180 Queen's Gate,\\
South Kensington Campus, London SW7, United Kingdom}
\ead{k.kourliouros10@imperial.ac.uk}

\begin{abstract}
In this note we give a positive answer to a question asked by Y. Colin de Verdi\`ere concerning the converse of the following theorem, due to A. N. Varchenko: two germs of volume forms are equivalent with respect to diffeomorphisms preserving a germ of an isolated hypersurface singularity, if their difference is the differential of a form whose restriction on the smooth part of the hypersurface is exact.  

\vspace{0.3cm}

\end{abstract}


\maketitle

\section{Introduction-Main Results}

In this paper we will give a positive answer to a question asked by Y. Colin de Verdi\`ere in \cite{C1} which was formulated as follows: suppose that two germs of symplectic forms at the origin of the plane are equivalent with respect to a diffeomorphism preserving a plane curve germ with an isolated singularity at the origin. Is it true that their difference is the differential of a 1-form whose restriction on the smooth part of the curve is exact? This question asks for the validity of the converse to a general normal form theorem in Lagrangian singularity theory  according to which: two germs of symplectic structures are equivalent with respect to diffeomorphisms preserving a Lagrangian variety if their difference is the differential of a 1-form whose restriction on the smooth part of the variety is exact. The proof of this theorem can be easily deduced from the reasoning in A. B. Givental's paper \cite{Gi} using Moser's homotopy method. It holds in any dimension and for arbitrary Lagrangian singularities. It's converse though is not so easy to deduce; as it turns out, the main difficulty comes from the fact that the singularities of Lagrangian varieties in dimension higher than two are non-isolated (c.f. \cite{Gi}, \cite{Sev}) and their cohomology can be rather complicated. On the other hand, for the 2-dimensional case (where the Lagrangian singularities are indeed isolated) the normal form theorem stated above can be viewed as a special case of a general theorem obtained by A. N. Varchenko in \cite{Var} concerning the normal forms of germs of (powers of) volume forms with respect to an isolated hypersurface singularity. Here we will prove a converse to Varchenko's normal form theorem, which trivially answers Verdi\`ere's question, and it can be formulated as follows:
\begin{thm}
\label{t}
Suppose that two germs of volume forms are equivalent with respect to a diffeomorphism preserving a germ of an isolated hypersurface singularity. Then their difference is the differential of a form whose restriction on the smooth part of the hypersurface is exact.
\end{thm}    

The method of proof is as follows: we first prove the theorem in the formal category. For this we use a formal interpolation lemma for the elements of the isotropy group of an isolated hypersurface singularity (Lemma \ref{l1}) which is a variant of the one presented by J. -P. Fran\c{c}oise in \cite{F} and relies in a general interpolation method obtained by S. Sternberg \cite{Ste}. Then we pass to the analytic category using a comparison theorem between the corresponding de Rham cohomologies in the formal and analytic categories (Lemma \ref{l2}). This is analogous to the well known Bloom-Brieskorn theorem \cite{B} for the de Rham cohomology of an analytic space with isolated singularities. But in contrast to the ordinary Bloom-Brieskorn theorem where the cohomology of the complex of K\"ahler differentials is considered, we need to consider instead the cohomology of the so called Givental complex, i.e. the complex of germs of holomorphic forms modulo those that vanish on the smooth part of the hypersurface (which naturally appears in the statements of the theorems above).

\section{De Rham Cohomology of an Isolated Hypersurface Singularity and an Analog of the Bloom-Brieskorn Theorem}

Let $f:(\mathbb{C}^{n+1},0)\rightarrow (\mathbb{C},0)$ be a germ of a holomorphic function with an isolated singularity at the origin and let $(X,0)=\{f=0\}$ be the corresponding hypersurface germ, zero level set of $f$ (we will suppose throughout that the germ $(X,0)$ is reduced). To the germ $(X,0)$ we may associate several complexes of holomorphic forms, quotients of the complex $\Omega^{\bullet}$ of germs of holomorphic forms at the origin of $\mathbb{C}^{n+1}$, the ``largest'' one being the so called complex of K\"ahler differentials:
\[\Omega^{\bullet}_{X,0}=\frac{\Omega^{\bullet}}{df\wedge \Omega^{\bullet-1}+f\Omega^{\bullet}},\]
where the differential is induced by the differential in $\Omega^{\bullet}$ after passing to quotients. The cohomologies of this complex are finite dimensional vector spaces and they have being computed by E. Brieskorn in \cite{B}. In particular, along with the results of M. Sebastiani \cite{S} it follows that:
\begin{equation}
\label{ck}
H^p(\Omega^{\bullet}_{X,0})=\left\{\begin{array}{cl}
							 \mathbb{C}, & p=0,\\
							 0, & 0<p<n, p>n \\
							 \mathbb{C}^{d}, & p=n,\\
							 \end{array} \right..
\end{equation}
The number $d$ can be interpreted as the degree of non-quasihomogeneity of the germ $f$, i.e.
\[d=\mu-\tau,\]
where $\mu$ is the Milnor number and $\tau$ is the Tjurina number of the singularity $f$:
\[\mu=\dim_{\mathbb{C}}\frac{\Omega^{n+1}}{df\wedge \Omega^n}, \hspace{0.3cm} \tau=\dim_{\mathbb{C}}\frac{\Omega^{n+1}}{df\wedge \Omega^n+f\Omega^{n+1}},\]
\[d=\dim_{\mathbb{C}}\frac{df\wedge \Omega^n+f\Omega^{n+1}}{df\wedge \Omega^n}.\]
Indeed, it is a result of K. Saito \cite{Sa} according to which $f$ is equivalent to a quasihomogeneous germ if and only if it belongs to its gradient ideal, i.e. $f\Omega^{n+1}\subset df\wedge \Omega^n$. 

Denote now by $X^*=X\setminus 0$ the smooth part of the hypersurface $X$. In \cite{Fe}, A. Ferrari introduced another important complex associated to $X$ which is the quotient complex of $\Omega^{\bullet}$ modulo the subcomplex $\Omega^{\bullet}(X^*)$ which consists of forms whose restriction on the smooth part $X^*$ of $X$ is identically zero:
\[\tilde{\Omega}^{\bullet}_{X,0}=\frac{\Omega^{\bullet}}{\Omega^{\bullet}(X^*)}.\]
This complex was also used extensively by A. B. Givental in \cite{Gi} and is called the Givental complex in \cite{He}. We adopt the same notation here as well.  As it is easy to see there is an identification of the complex of K\"ahler differentials with the Givental complex on the smooth part $X^*$ and thus there is a short exact sequence of complexes:
\begin{equation}
\label{ses0}
0\rightarrow T^{\bullet}_{X,0}\rightarrow \Omega^{\bullet}_{X,0}\rightarrow \tilde{\Omega}^{\bullet}_{X,0}\rightarrow 0,
\end{equation}
where $T^{\bullet}_{X,0}$ is the torsion subcomplex of $\Omega^{\bullet}_{X,0}$ (here is where we need $(X,0)$ to be reduced). Indeed any torsion element vanishes on the smooth part $X^*$ and thus the complex $T^{\bullet}_{X,0}$ is contained in the kernel of the natural projection $\Omega^{\bullet}_{X,0}\rightarrow \tilde{\Omega}^{\bullet}_{X,0}$. 

In \cite{G}, G. M. Greuel studied the relationship of the Givental and K\"ahler complexes in the general case where $(X,0)$ defines an $n$-dimensional isolated complete intersection singularity (embedded in some $\mathbb{C}^{m}$). He proves that:
\[T^{p}_{X,0}=0, \hspace{0.3cm}  p< n, \]
\[T^p_{X,0}=\Omega^{p}_{X,0}, \hspace{0.3cm} p>n,\]
and also:
\[H^p(\Omega^{\bullet}_{X,0})=0,\hspace{0.3cm} 0<p<n,\]
\[H^p(\tilde{\Omega}^{\bullet}_{X,0})=0, \hspace{0.3cm} p\neq 0,n.\]
Thus, in the particular case where $(X,0)$ is an isolated hypersurface singularity we obtain the following analog of the Brieskorn-Sebastiani result (\ref{ck}) for the cohomology of the Givental complex:
\begin{prop}
\label{p1}
\[H^p(\tilde{\Omega}^{\bullet}_{X,0})=\left\{\begin{array}{cl}
							 \mathbb{C}, & p=0,\\
							 0, & 0<p<n, p>n \\
							 \mathbb{C}^{d}, & p=n,\\
							 \end{array} \right.,\]
where $d=\mu-\tau$ is the degree of non-quasihomogeneity of the germ $f$.
\end{prop} 
\begin{proof}
It suffices only to show the following equality (the zero cohomology is trivial):
\[H^n(\tilde{\Omega}^{\bullet}_{X,0})=\mathbb{C}^d.\]
This in turn has been proved by A. N. Varchenko in \cite{Var}. Here we will give an alternative, simple proof, which is distilled from \cite{C1}.  To the germ $f$ we associate the Brieskorn module as in \cite{B}:
\[H_f''=\frac{\Omega^{n+1}}{df\wedge d\Omega^{n-1}}.\]
According to the Sebastiani theorem \cite{S} this is a free module of rank $\mu$ over $\mathbb{C}\{f\}$ and thus the quotient 
\[\frac{H''_f}{fH''_f}=\frac{\Omega^{n+1}}{df\wedge d\Omega^{n-1}+f\Omega^{n+1}}\] 
is a $\mu$-dimensional $\mathbb{C}$-vector space. Denote now by 
\[\mathcal{Q}_{X,0}=\frac{\Omega^{n+1}}{df\wedge \Omega^n+f\Omega^{n+1}}\]
the space of deformations of the germ $(X,0)$. By the fact that $df\wedge d\Omega^{n-1}+f\Omega^{n+1}\subseteq df\wedge \Omega^n+f\Omega^{n+1}$ there is a natural projection:
\[\frac{H''_f}{fH''_f}\stackrel{\pi}{\rightarrow}\mathcal{Q}_{X,0},\]
whose kernel:
\[\ker{\pi}=\frac{df\wedge \Omega^n}{df\wedge d\Omega^{n-1}+f\Omega^{n+1}}\]
is a priori a $d=\mu-\tau$-dimensional vector space. Now, the $n$-th cohomology of the Givental complex is:
\[H^n(\tilde{\Omega}^{\bullet}_{X,0})=\frac{\tilde{\Omega}^n_{X,0}}{d\tilde{\Omega}^{n-1}_{X,0}}=\frac{\Omega^n}{\Omega^n(X^*)+d\Omega^{n-1}},\]
where:
\[\Omega^n(X^*)=\{\alpha \in \Omega^n/df\wedge \alpha \in f\Omega^{n+1}\}.\]
It follows from this that 
\[\ker{\pi}=df\wedge H^n(\tilde{\Omega}^{\bullet}_{X,0})\]
and thus there is a short exact sequence:
\begin{equation}
\label{ses1}
0\rightarrow H^n(\tilde{\Omega}^{\bullet}_{X,0})\stackrel{df\wedge}{\rightarrow}\frac{H''_f}{fH''_f}\stackrel{\pi}{\rightarrow}\mathcal{Q}_{X,0}\rightarrow 0.
\end{equation} 
This proves that indeed $H^n(\tilde{\Omega}^{\bullet}_{X,0})=\mathbb{C}^d$ as was asserted. 
\end{proof}

It follows from the proposition above along with (\ref{ck}) that there is an isomorphism of vector spaces:
\[H^{\bullet}(\Omega^{\bullet}_{X,0})\cong H^{\bullet}(\tilde{\Omega}^{\bullet}_{X,0}).\]
Thus we may formulate  the following version of the Poincar\'e lemma for the germ $(X,0)$:
\begin{cor}[c.f. \cite{Gi} for $n=1$]
\label{c1}
The germ $(X,0)$ is quasihomogeneous if and only if its Givental (or K\"ahler) complex is acyclic (except in zero degree).
\end{cor}

Finally, we will need  the following analog of the Bloom-Brieskorn theorem \cite{B}, which is a comparison of the cohomologies of the analytic and formal Givental complexes. The proof we will give below is in fact a simple variant of the one presented in \cite{B}. Moreover, the fact that $(X,0)$ is an isolated hypersurface singularity plays no significant role; the same proof holds for any analytic space, as long as its singularities are isolated.
\begin{lem}
\label{l2}
Let $\hat{\tilde{\Omega}}^{\bullet}_{X,0}$ be the formal completion of the Givental complex. Then the natural inclusion $\tilde{\Omega}^{\bullet}_{X,0}\hookrightarrow \hat{\tilde{\Omega}}^{\bullet}_{X,0}$ induces an isomorphism of finite dimensional vector spaces:
\[H^{\bullet}(\tilde{\Omega}^{\bullet}_{X,0})\cong H^{\bullet}(\hat{\tilde{\Omega}}^{\bullet}_{X,0}).\]
\end{lem} 
\begin{proof}
Following \cite{B} let $\pi:Y\rightarrow X$ be a resolution of singularities in the sense of Hironaka and denote by $A=\pi^{-1}(0)$ the exceptional set, which we may suppose it is given by some equations $y_1\cdots y_r=0$. Let $\Omega^{\bullet}_{Y}$ be the complex of holomorphic forms on $Y$ and let $\Omega^{\bullet}_Y|_{A}$ be its restriction on $A$. Let also 
\[\hat{\Omega}^{\bullet}_{Y}=\lim_{\underset{k}{\leftarrow}} \frac{\Omega^{\bullet}_Y}{\frak{m}^k\Omega^{\bullet}_Y}.\]
Consider now the direct image sheaf $R^0\pi_*\Omega^{\bullet}_{Y}$ (this is also called the Noether complex). Since the map $\pi$ is proper this is a coherent sheaf (by Grauert's coherence theorem), which away from the singular point $0$ it can be identified with the Givental complex: $R^0\pi_*\Omega^{\bullet}_Y|_{X^*}\cong \tilde{\Omega}^{\bullet}_{X^*}$. In particular there is an inclusion $j:\tilde{\Omega}^{\bullet}_{X}\rightarrow R^0\pi_*\Omega^{\bullet}_Y$ whose cokernel is concentrated at the singular point $0$ and it is thus finite dimensional. Consider now the formal completion of the above complexes. It gives a commutative diagram:
\begin{equation}
\begin{CD}
\tilde{\Omega}^{\bullet}_{X,0} @>j>>H^0(A, \Omega^{\bullet}_Y|_{A}) \\
 @VVV          @VVV         \\
\hat{\tilde{\Omega}}^{\bullet}_{X,0} @>\hat{j}>>H^0(A, \hat{\Omega}^{\bullet}_Y)  \\
\end{CD}
\end{equation}
where of course $H^0(A, \Omega^{\bullet}_Y|_{A}) \cong (R^0\pi_*\Omega^{\bullet}_Y)|_0$ and $\hat{j}$ is the formal completion of the inclusion $j$. Indeed, this follows from the fact (c.f. \cite{B} and the corresponding references therein):
\[H^0(A,\hat{\Omega}^{\bullet}_Y)\cong \lim_{\underset{k}{\leftarrow}}H^0(A, \frac{\Omega^{\bullet}_Y}{\frak{m}^k\Omega^{\bullet}_Y})\cong \lim_{\underset{k}{\leftarrow}}\frac{H^0(A, \Omega^{\bullet}_Y|_{A})}{\frak{m}^kH^0(A, \Omega^{\bullet}_Y|_{A})}.\]
Now, since the completion functor is exact and by the fact that the cokernel of $j$ is already complete (by finite dimensionality), it follows that 
\[\text{Coker}j\cong \text{Coker}\hat{j}.\] 
Thus, in order to show the theorem starting from the commutative diagram above, it suffices to show the isomorphism:
\[H^{\bullet}(H^0(A,\Omega^{\bullet}_{Y}|_{A}))\cong H^{\bullet}(H^0(A,\hat{\Omega}^{\bullet}_Y)).\]
This is proved in turn in \cite{B} (points (b)-(d), pp. 140-142).
\end{proof}
\begin{rem}
For the hypersurface case, there is a simple alternative proof of the above lemma, only for the $n$th-cohomology of the Givental complex, without using resolution of singularities: let $\hat{H}''_f$ be the formal completion of the Brieskorn module with respect to the $\frak{m}$-adic topology. Then, by the regularity of the Gauss-Manin connection and the properties of its analytical index \cite{Mal}, there is an isomorphism of $\mathbb{C}[[f]]$-modules\footnote{or equivalently by the Bloom-Brieskorn theorem \cite{B}, but this uses again resolution of singularities.}:
\[\hat{H}''_f\cong H''_f\otimes_{\mathbb{C}\{f\}} \mathbb{C}[[f]]\]
and thus the quotient
\[\frac{\hat{H}''_f}{f\hat{H}''_f}=\frac{\hat{\Omega}^{n+1}}{df\wedge d\hat{\Omega}^{n-1}+f\hat{\Omega}^{n+1}}\]
is again a $\mu$-dimensional vector space. The space of deformations $\mathcal{Q}_{X,0}$ of the germ $(X,0)$ is finite dimensional and thus it is already complete:
\[\mathcal{Q}_{X,0}\cong \hat{\mathcal{Q}}_{X,0}.\]
Following the construction presented in the proof of Proposition \ref{p1} for the cohomology $H^n(\tilde{\Omega}^{\bullet}_{X,0})$ we obtain again a short exact sequence:
\[0\rightarrow H^n(\hat{\tilde{\Omega}}^{\bullet}_{X,0})\stackrel{df\wedge}{\rightarrow}\frac{\hat{H}''_f}{f\hat{H}''_f}\stackrel{\pi}{\rightarrow}\hat{\mathcal{Q}}_{X,0}\rightarrow 0.\]
The proof of the isomorphism 
\begin{equation}
\label{iso}
H^n(\tilde{\Omega}^{\bullet}_{X,0})\cong H^n(\hat{\tilde{\Omega}}^{\bullet}_{X,0})
\end{equation}
follows then immediately by comparing the short exact sequence above with the analytic one (\ref{ses1}). 
\end{rem}

\section{An Interpolation Lemma for the Isotropy Group of a Hypersurface Singularity}

Let $\mathcal{R}_{X,0}$ be the isotropy group of the germ $(X,0)$, i.e. the group of germs of diffeomorphisms at the origin tangent to the identity and preserving the hypersurface $X=\{f=0\}$. It means that for every $\Phi \in \mathcal{R}_{X,0}$ there exists an invertible function germ $g \in \mathcal{O}$ such that the following hold:
\[\Phi(x)=x \hspace{0.15cm} \text{mod} \hspace{0.15cm} \frak{m}^2, \hspace{0.3cm} g(x)=1\hspace{0.15cm} \text{mod}\hspace{0.15cm} \frak{m},\]
\[\Phi^*f=gf.\]

We will need the following interpolation lemma for the group $\mathcal{R}_{X,0}$ which is a simple variant of the one presented by J. -P. Fran\c{c}oise in \cite{F} and it relies in a general method obtained by S. Sternberg in \cite{Ste}. It can be also generalised without difficulty to any germ of an analytic subset $(X,0)$ (whose singularities can be arbitrary).

\begin{lem}
\label{l1}
Any diffeomorphism $\Phi \in \mathcal{R}_{X,0}$  can be interpolated by a 1-parameter family of formal diffeomorphisms $\Phi_t \in \hat{\mathcal{R}}_{X,0}$, i.e. there exists a family of formal function germs $g_t \in \hat{\Omega}^0$ such that:
\[\Phi_0=Id, \hspace{0.3cm} \Phi_1=\Phi,\]
\[g_0=1, \hspace{0.3cm} g_1=g,\]
\[\Phi_t^*f=g_tf.\]
\end{lem}
\begin{proof}
Denote by $(x_1,...,x_{n+1})$ the coordinates at the origin and let $x^{\beta}=x_1^{\beta_1}...x_{n+1}^{\beta_{n+1}}$, $\beta=(\beta_1,...,\beta_{n+1})\in \mathbb{N}^{n+1}$, $|\beta|=\sum_{i=1}^{n+1}\beta_i$. Let 
\[\Phi_i(x)=x_i+\sum_{j}\sum_{|\beta|=j}\phi_{i,\beta}x^{\beta}, \hspace{0.3cm} i=1,...,{n+1}\]
be the components of $\Phi$. We will find the interpolation $\Phi_t$ with components in the form:
\[\Phi_{t,i}(x)=x_i+\sum_j\sum_{|\beta|=j}\phi_{i,\beta}(t)x^{\beta}, \hspace{0.3cm} i=1,...,{n+1}\]
as solution of the differential equation:
\begin{equation}
\label{fe}
\Phi'_t=\Phi'_0\circ \Phi_t, 
\end{equation}
with boundary conditions $\Phi_0=Id$, $\Phi_1=\Phi$ (c.f. \cite{Ste}). We can can do this by induction on $j$ and we may assume that the $\phi_{i,\beta}$ are already known for $j\leq k-1$. Then, for $j=k$, equation (\ref{fe}) implies:
\[\phi'_{i,\beta}(t)=\phi'_{i,\beta}(0)+\psi_{i,\beta}(t),\]
where the functions $\psi_{i,\beta}(t)$ are known by induction and they vanish at zero. Integration then gives:
\[\phi_{i,\beta}(t)=\phi'_{i,\beta}(0)t+\int_0^t\psi_{i,\beta}(\tau)d\tau.\]
Obviously the initial condition $\phi_{i,\beta}(0)=0$ is satisfied, and it suffices to choose the $\phi'_{i,\beta}(0)$ such that the boundary condition $\phi_{i,\beta}(1)=\phi_{i,\beta}$ is satisfied as well. Now, by the fact that the family $\Phi_t$ is an interpolation of $\Phi$, we may choose an interpolation $g_t$ of $g$:
\[g_t(x)=g(0)+\sum_{|\beta|\geq 1}g_{\beta}(t)x^{\beta},\]
satisfying the required assumptions (recall that $g(0)=1$) and such that $\Phi_t^*f=g_tf$ for all integer values of $t$. In fact, the coefficients of $\Phi_t$ are polynomials in $t$, and choosing the interpolation $g_t$ with polynomial coefficients in $t$ as well (linear in $t$ for example), it follows that for any $k$ fixed, the homogeneous part in the Taylor expansion of $\Phi_t^*f-g_tf$ is a polynomial in $t$ which vanishes for all integer values of $t$. Thus, it vanishes for all real $t$ as well and this finishes the proof of the lemma.  
\end{proof}

\section{Proof of the Theorem}
We will prove here Theorem \ref{t} which can now be restated in the following form:
\begin{thm}
\label{t2}
Let $\omega$ and $\omega'$ be two germs of volume forms which are $\mathcal{R}_{X,0}$-equivalent. Then there exists an $n$-form $\alpha$ such that $\omega-\omega'=d\alpha$ and $[\alpha]=0$ in $H^n(\tilde{\Omega}^{\bullet}_{X,0})$.
\end{thm}
\begin{proof}
Consider first the $n$-form $\alpha$ defined by $\omega-\omega'=d\alpha$ (Poincar\'e lemma) and let $\Phi \in \mathcal{R}_{X,0}$ be the diffeomorphism providing the equivalence: $\Phi^*\omega'=\omega$.
It follows that
\begin{equation}
\label{c1}
\omega-\Phi^*\omega=d\alpha
\end{equation}
holds in $\Omega^{n+1}$. Interpolate now $\Phi$ by the 1-parameter family of formal diffeomorphisms $\Phi_t\in \hat{\mathcal{R}}_{X,0}$ as in Lemma \ref{l1} above. We have that:
\[\omega-\Phi^*\omega=\int_0^1\frac{d}{dt}\Phi_t^*\omega dt=\int_0^1\Phi_t^*(L_{\hat{v}}\omega) dt=\]
\[=\int_0^1\Phi_t^*d(\hat{v}\lrcorner \omega)dt=d\int_0^1\Phi_t^*(\hat{v}\lrcorner \omega)dt,\]
holds in $\hat{\Omega}^{n+1}$, where $\hat{v}$ is the 1-parameter family of formal vector fields generating $\Phi_t$: $\exp{t\hat{v}}=\Phi_t$. Thus, in $\hat{\Omega}^{n+1}$ we may write:
\begin{equation}
\label{c2}
\omega-\Phi^*\omega=d\hat{\alpha},
\end{equation}
where the formal $n$-form $\hat{\alpha}$ is defined by:
\[\hat{\alpha}=\int_0^1\Phi_t^*(\hat{v}\lrcorner \omega)dt+d\hat{h},\]
for some formal $(n-1)$-form $\hat{h}$. Now, since $\Phi_t$ preserves the germ $(X,0)$ for all $t$ and $\hat{v}$ is tangent to its smooth part, it follows that $\hat{\alpha}|_{X^*}=d\hat{h}|_{X^*}$, i.e. that $[\hat{\alpha}]=0$ in $H^n(\hat{\tilde{\Omega}}^{\bullet}_{X,0})$. View now the relation (\ref{c1}) as a relation in $\hat{\Omega}^{n+1}$. By comparing it with the relation (\ref{c2}) we obtain $\alpha=\hat{\alpha}+d\hat{g}$ for some formal $(n-1)$-form $\hat{g}$ and thus $[\alpha]=[\hat{\alpha}]=0$ in $H^n(\hat{\tilde{\Omega}}^{\bullet}_{X,0})$ as well.  By the the Bloom-Brieskorn Lemma \ref{l2} and in particular by the isomorphism (\ref{iso}) we finally obtain that $[\alpha]=0$ in $H^n(\tilde{\Omega}^{\bullet}_{X,0})$ and this finishes the proof of the theorem.
\end{proof}

\section*{Acknowledgements}
The author is grateful to the staff of ICMC-USP (Instituto de Ci\^encias Matem\'aticas e de Computa\c{c}\~ao-Universidade de S\~ao Paulo) for their hospitality, where the paper was written. 

The author is also grateful to Mauricio D. Garay for his valuable comments on a previous version of the paper, as well as to Christian Sevenheck for his interest in the work and for providing directions concerning G. M. Greuel's work.  

 ``This research has been supported (in part) by EU Marie-Curie IRSES Brazilian-European partnership in Dynamical Systems (FP7-PEOPLE-2012-IRSES 318999 BREUDS)''.

\section*{Bibliography}


\begin{thebibliography}{55}     
\bibitem{B} E. Brieskorn, \textit{Die Monodromie der Isolierten Singularit\"aten von H\"yperfl\"anchen}, Manuscripta Math., 2 (1970), 103-161 
\bibitem{C1} Y. Colin de Verdiere, \textit{Singular Lagrangian manifolds and Semi-Classical Analysis}, Duke Math. Journal 116, (2011), 263-298
\bibitem{Fe} A. Ferrari, \textit{Cohomology and Holomorphic Differential Forms on Complex Analytic Spaces}, Ann. Scuola Norm. Sup. Pisa, Tome 24, No. 1, (1970), 65-77 
\bibitem{F} J. P. Francoise, \textit{Relative Cohomology and Volume Forms}, Singularities, Banach Center Publications, 20, (1988), 207-222
\bibitem{Gi} A. B. Givental, \textit{Singular Lagrangian Manifolds and their Lagrangian Mappings}, Itogi Nauki Tekh., Ser. Sovrem. Probl. Mat., Novejshie Dostizh. 33 (1988), 55-112, English translation: J. Soviet Math. 52.4 (1990), 3246-3278
\bibitem{G} G. M. Greuel, \textit{Der Gauss-Manin-Zusammenhang Isolierter Singularit\"aten von Vollst\"andigen Durchschnitten}, Math. Ann., (1975), 235-266
\bibitem{He} C. Hertling, \textit{Frobenious Manifolds and Moduli Spaces for Singularities}, Cambridge Tracts in Mathematics, 151, Cambridge University Press, (2002)
\bibitem{Mal} B. Malgrange, \textit{Int\'egrales Asymptotiques et Monodromie}, Ann. Scient. Ec. Norm. Sup., 7, (1974), 405-430 
\bibitem{Sa} K. Saito, \textit{Quasihomogene Isolierte Singularit\"aten von Hyperfl\"achen}, Invent. Math. 14 (1971), 123-142
\bibitem{S} M. Sebastiani, \textit{Preuve d'une Conjecture de Brieskorn}, Manuscripta Math., 2 (1970), 301-308
\bibitem{Ste} S. Sternberg, \textit{Infinite Dimensional Lie Groups and Formal Aspects of Dynamical Systems}, J. Math. Mec., 10, (1961), 451-474 
\bibitem{Sev} C. Sevenheck, \textit{Singularit\'es Lagrangiennes}, Phd Thesis, \'Ecole Polytechnique, (2003)
\bibitem{Var} A. N. Varchenko, \textit{Local Classification of Volume Forms in the Presence of a Hypersurface}, M. V. Lomonosov State University, Moscow. Translated from Funktsional'nyi Analiz i Ego Prilozheniya, Vol. 19, No. 4, (1985), 23-31
\end{thebibliography}
\end{document}